\documentclass{amsart}
\usepackage{amsmath,amsthm,latexsym,verbatim}
\usepackage[all]{xypic}
\usepackage{amssymb}
\usepackage{fancyhdr}
\usepackage{mathrsfs}
\usepackage{hyperref}
\usepackage{tikz}
\usetikzlibrary{arrows}

\theoremstyle{definition}

\newtheorem{Def}{Definition}[section]
\theoremstyle{plain}
\newtheorem{Thm}{Theorem}[section]
\newtheorem{Lem}[Thm]{Lemma}

\newtheorem{Prop}[Thm]{Proposition}

\newtheorem{rem}{Remark}[section]

\newcommand{\proj}{\textnormal{proj-lim}}

\newcommand{\R}{\mathbb{R}}

\newcommand{\C}{\mathbb{C}}

\newcommand{\Z}{\mathbb{Z}}
\newcommand{\N}{\mathbb{N}}
\newcommand{\Si}{\mathscr{S}}

\newcommand{\s}{\mathbf{s}}

\newcommand{\tr}{\textrm{tr}}
\newcommand{\op}{\textnormal{Op}}

\newcommand{\dbar}{\;{\mathchar'26\mkern-10mu  d}}
\newcommand{\limproj}{\textnormal{proj-lim}}

\fancypagestyle{plain}{
	\fancyhead{}
	
}

\author{Ubertino Battisti}

\address{Dipartimento di Matematica, Universit\`a di Torino, Italy}

\email{ubertino.battisti@unito.it}

\author{Sandro Coriasco}

\address{Dipartimento di Matematica, Universit\`a di Torino, Italy}

\email{sandro.coriasco@unito.it}

\author{Elmar Schrohe}

\address{Institut f\"ur Analysis, Leibniz Universit\"at Hannover, Germany}

\email{schrohe@math.uni-hannover.de}

\title{On a Class of Fourier Integral Operators\\on Manifolds with Boundary}



\def\sphi{{^*\!\Phi}}
\newcommand\norm[1]{\langle {#1} \rangle}
\def\SG{\mathrm{SG}}

\begin{document}

\begin{abstract}
We study a class of Fourier integral operators on compact manifolds with boundary $X$ and $Y$, associated with a natural class of symplectomorphisms $\chi\colon T^*Y\setminus 0\to T^*X\setminus 0$, namely, those which preserve the boundary. A calculus of Boutet de Monvel's type can be defined for such Fourier integral operators, and appropriate continuity properties established. 
One of the key features of this calculus is that the local representations of these operators
are given by operator-valued symbols acting on Schwartz functions or temperate distributions. Here we focus on properties of the corresponding local phase functions, which allow to prove this result in a rather straightforward way.

\end{abstract}

\maketitle

\section{Introduction}
In \cite{BCS14} we developed a Boutet de Monvel type calculus of (block matrices of) Fourier integral operators on compact manifolds with boundary. 
We recall the basic features: 
In the sequel we fix two compact $n$-dimensional manifolds $X$ and $Y$ with boundary and a symplectomorphism $$\chi: T^*Y \setminus 0 \to T^*X\setminus 0,$$ 
which is positively 
homogeneous of degree $1$ in the fibers and preserves the boundary, that is, 
$\pi_{\partial X}\circ\chi=b\circ\pi_{\partial Y}$, with a diffeomorphism 
$b\colon \partial Y\to\partial X$ and the canonical projections $\pi_{\partial X}$, $\pi_{\partial Y}$ at the boundaries.
By a variant of Moser's trick, cf.  \cite{DAS01}, Chapter 7, $\chi$ can be extended to a symplectomorphism $$\tilde{\chi}: T^*\widetilde{Y} \setminus 0 \to T^*\widetilde{X}\setminus 0,$$ 
where $\widetilde{Y}$ and $\widetilde{X}$ are neighborhoods of $X$ and $Y$, in  closed $n$-dimensional manifolds $X^\prime$ and  $Y^\prime$ containing
$X$ and $Y$, respectively. 
It turns out that the homogeneity of $\chi$, together with the fact  that it preserves the boundary, implies that $\chi$ induces a symplectomorphism 
$$\chi_\partial: T^*\partial Y \setminus 0 \to T^*\partial X\setminus 0 ,$$ 
which is the lift of the diffeomorphism $b: \partial Y \to \partial X$, cf.\ Lemma \ref{lem:chidelta}, below. 
We then considered truncated Fourier integral operators of the form
$$A^+=r^+A^\chi e^+.$$
Here, $e^+: C^\infty(Y)\to C^\infty(\tilde Y)$ is the operator of extension by zero, 
$A^\chi$ is a Fourier integral operator whose kernel is an Lagrangian distribution 
associated with the graph of $\tilde \chi$ and $r^+$ denotes the restrictions of distributions 
on $\tilde X$ to $\mathrm{int} X $. It is well known that the elements of the subclass of Fourier integral operators associated with graphs of symplectomorphisms have many good properties, which, in a sense, make their calculus similar to the calculus of pseudodifferential operators, see L. H\"ormander \cite{HO04}. 
These similarities have been used to great advantage in \cite{BCS14}. In particular, the phase functions 
$\phi(x,y,\xi)=\psi(x,\xi)-\langle y,\xi\rangle$  
of such operators have a very special structure near the boundary, due to the fact that the boundary is preserved. 
In order to ensure good mapping properties, we moreover made the assumption that all components of the symplectomorphism $\tilde{\chi}$ (that is, the first derivatives of $\psi$), satisfy \textit{transmission property} at the boundary. 

We remark that the above assumptions are all natural. 
In a sense, they provide one of the simplest extensions of the concept of Fourier integral operator from the case of closed manifolds to the case of compact manifolds with boundary. 
In fact, we obtain an extension of the calculus of pseudodifferential boundary value problems defined by  Boutet de Monvel \cite{BU71}, which in our setting corresponds to $X=Y$ and $\chi=id$.  
In addition to being of interest in itself, our class of operators provides an analytic framework which can be used to study two different problems. 
The first is an index problem, analogous to the one considered by A.\ Weinstein in \cite{WE75}.
The second is the problem of classifying -- similarly as this was done by  J.J. Duistermaat and I. Singer in \cite{DS76} for the case of closed manifolds --  the order-preserving isomorphisms between the Boutet de Monvel algebras on $X$ and $Y$, using elements of our class of Fourier integral operators. 
We also note that our framework can be considered complementary to that 
introduced by A. Hirschowitz and A. Piriou in \cite{HP77}, who studied the transmission property for Fourier distributions conormal to hypersurfaces in $T^*X\setminus 0$.

Here we focus on one of the key features of the calculus, namely, showing how the properties of the symplectomorphism
$\chi$ reflect into those of the phase functions, so that the local representations of the Fourier integral operators associated with 
$\mathrm{graph}(\tilde{\chi})$ can be considered as operator-valued symbols. More precisely, let $A^+$ be as above and $(A^*)^+= r^+(A^\chi)^*e^+$ with the formal $L^2$ adjoint $(A^\chi)^*$ of $A^\chi$. 
Under our hypotheses, it is possible to prove that  
\begin{eqnarray}\label{eq:map}
A^+\colon C^\infty(Y) \to C^\infty(X)\text{ and  }(A^*)^+\colon C^\infty(X)\to C^\infty(Y)
\end{eqnarray}
continuously, in analogy with the corresponding results in the Boutet de Monvel calculus.
We will show that, for a symbol $a\in S^m(\R^n\times \R^n)$ satisfying the transmission condition, the operator family 
\begin{equation}\label{eq:opan}
A_n^\chi=
  \op^\psi_{n}(a): u \mapsto \int e^{i (\psi(x', x_n, \xi', \xi_n)- \psi_\partial(x',\xi'))} a(x', x_n, \xi', \xi_n) \widehat{u}(\xi_n) \dbar \xi_n,
 \end{equation}
describing the action of $A^\chi$ in the normal direction, is an operator-valued symbol acting from $\Si(\R)$ to itself and from $\Si^\prime(\R)$
to itself. We refer the reader to E. Schrohe \cite{SC01} and B.-W. Schulze \cite{SC98} for the precise definitions of the involved semigroup 
actions and of operator-valued symbols. 

In \eqref{eq:opan}, $\psi$ is a phase function which locally represents $\tilde{\chi}$ close to the boundaries, and $a$ is a symbol in  $S^m(\R^n\times \R^n)$. Both $a$ and $\psi$ are required to satisfy the transmission condition;  more details will be given, below. 
The phase $\psi_\partial$ represents the symplectomorphism $\chi_\partial$ between the cotangent bundles of the boundaries 
induced by $\chi$, that is 
$$\psi_\partial(x^\prime,\xi^\prime)=\psi(x^\prime,0,\xi^\prime,\xi_n).$$ 
Our main results are the following Theorems \ref{thm:main} and \ref{thm:main2}.
In the former, we prove that the phase function in $(x_n,\xi_n)\in\R^2$ appearing in the operator-valued estimates of $A_n^\chi$ is a regular SG phase function, in the sense of S. Coriasco \cite{CO99}, see also \cite{CO99b}. In the latter, we prove that $A_n^\chi$ is indeed an operator-valued symbol belonging to $S^m(\R^{n-1}, \R^{n-1}, \Si(\R), \Si(\R))$ and to $S^m(\R^{n-1}, \R^{n-1}, \Si^\prime(\R), \Si^\prime(\R))$,
as a corollary of Theorem \ref{thm:main}.

We recall the following definition, streamlined to our purposes: 

\begin{Def}A smooth function $\Phi$ on $\R\times \R$ is a regular SG phase function, provided
it has the following three properties
{\renewcommand{\labelenumi}{{\rm (P\arabic{enumi})}}
\begin{enumerate}
\item $\Phi\in S^{1,1}(\R\times\R)$, i.e. for all $a,\alpha\in\Z_+$ there exist 
$C_{a\alpha}>0$ such that 
$$|D^a_t D^\alpha_\tau\sphi(t,\tau)|\le C_{a\alpha}\norm{t}^{1-a}\norm{\tau}^{1-\alpha}, \quad t,\tau \in \R.$$

\item There exist $c,C>0$ such that 
$$c\norm{\tau}\le\norm{\Phi^\prime_t(t,\tau)}\le C\norm{\tau},
	c\norm{t}\le\norm{\Phi^\prime_\tau(t,\tau)}\le C\norm{t}, \quad t,\tau\in \R.$$

\item There exists an $ \varepsilon>0$ such that $$|\Phi^{\prime\prime}_{t\tau}(t,\tau)|\ge\varepsilon,\quad  t,\tau\in\R.$$
\end{enumerate}
}
\end{Def}

In the sequel we denote by $\omega\in C^\infty(\R)$ an even function,  
non-increasing on $\R_+$, with $\omega\equiv1$ on $[0,1/2]$ and  
$\omega(t)=0$ for $t\ge 1$. 
We write $\omega_k(t) = \omega(t/k)$, $k>0$. 

\noindent
\begin{Thm}\label{thm:main}
Let $\Omega'$ be a connected open subset of $\R^{n-1}$ and $\Omega= \Omega^\prime\times \, ]-1,1[$. Write coordinates  in $\Omega $ as $x=(x',x_n)$ with $x'\in \Omega'$ and $x_n\in\, ]-1,1[$.  
By $\xi=(\xi',\xi_n)\in\R^n\setminus\{0\}$ we denote the corresponding covariable.
We let 
$$\varphi(x,\xi)=\psi(x,\xi)-\psi_\partial(x^\prime,\xi^\prime)$$ 
and, for $k,K>0$
\[
	\Phi_{x^\prime,\xi^\prime}(x_n,\xi_n)=\Phi(x,\xi)=\omega_k(x_n)\varphi(x,\xi)+(1-\omega_k(x_n))K\cdot x_n\cdot\xi_n.
\]
%
Define\footnote{Here we assume, as it is of course possible without loss of generality, that $\psi$ is well-defined for $x_n\in[-2k,2k]$, $k>0$ small enough, and extend the first summand in the second line of the definition of $\sphi$ identically equal to $0$ when 
${|t|}/{\norm{\xi^\prime}}\ge2k$.}, for $(t,\tau)\in\R^2$,
\begin{equation}
\label{eq:1}
\begin{aligned}
	\sphi(t,\tau)&=\sphi_{x^\prime,\xi^\prime}(t,\tau)=\Phi\left(x^\prime,{t}/{\norm{\xi^\prime}},\xi^\prime,\tau\norm{\xi^\prime}\right)
	\\
	&=\omega_k\left(\frac{t}{\norm{\xi^\prime}}\right)\varphi\left(x^\prime,\frac{t}{\norm{\xi^\prime}},\xi^\prime,\tau\norm{\xi^\prime}\right)
	+ \left[1-\omega_k\left(\frac{t}{\norm{\xi^\prime}}\right)\right]K\cdot t\cdot\tau.
\end{aligned}
\end{equation}
Then $\sphi$ is a regular SG phase function, provided $K$ is large enough and $k$ 
is small enough. 
%
%
%
%
%
Moreover, the constants $C_{a\alpha}$, $a,\alpha\in\Z_+$, $c,C,\varepsilon$, appearing in the estimates {\rm (P1), (P2),} and {\rm (P3)} above do not depend on 
$(x^\prime,\xi^\prime)\in U^\prime\times(\R^{n-1}\setminus\{0\})$ for  $U^\prime\subset\subset\Omega^\prime$.
\end{Thm}

\begin{Thm}\label{thm:main2}
	Under the hypotheses of Theorem \ref{thm:main} the operator 
	$A_n^\chi$, defined in \eqref{eq:opan}, satisfies
	\begin{equation}\label{eq:opsymb}
	A_n^\chi \in S^m(\R^{n-1}, \R^{n-1}, \Si(\R), \Si(\R))
	\text{ and }
	A_n^\chi \in S^m(\R^{n-1}, \R^{n-1}, \Si^\prime(\R), \Si^\prime(\R)),
	\end{equation}
	for any $a\in S^m(\R^n\times \R^n)$ satisfying the transmission condition, supported in a suitably small collar 
	neighborhood of the boundary.
\end{Thm}

We remark that the above theorems are essential to achieve the calculus for the Fourier integral operators of Boutet de Monvel type that we consider in \cite{BCS14}. They are needed, in particular, to show that the operator
\begin{equation}\label{eq:opan2}
A_n^+=
  r^+\op^\psi_{n}(a)e^+: u \mapsto r^+\int e^{i [\psi(x', x_n, \xi', \xi_n)- \psi_\partial(x',\xi')]} a(x', x_n, \xi', \xi_n) \widehat{e^+u}(\xi_n) \dbar \xi_n
\end{equation}
describing the action of $A^+$ in the normal direction, is also an operator-valued symbol,
acting from $\Si(\R_+)$ to itself and from $\Si^\prime(\R_+)$ to itself. Here,
\[
	\Si(\R_+)=\{h=g|_{\R_+}\colon g\in\Si(\R)\},
\]
endowed with its natural topology, and $\Si^\prime(\R_+)$ is the dual of $\Si(\R_+)$.

While this result may be  expected, its proof is rather delicate and requires a careful analysis of the properties of the kernels of the operators involved. Note also that, in strong contrast with the corresponding result 
for the Boutet de Monvel calculus, it is by no means true that $A^+_n$ belongs to
$S^m(\R^{n-1}, \R^{n-1};$ $ H^s(\R_+), H^{s-m}(\R_+))$ for each $s\in\R$,
see the counterexample in \cite{BCS14}.

\medskip

The paper is organized as follows. In Section \ref{sec2} we recall the main properties of the local phase
functions associated with $\tilde \chi$  near the boundaries.
In Section \ref{sec3} we describe the basic elements of the theory of the Fourier integral operators we consider. The material in Sections \ref{sec2} and \ref{sec3} is taken from \cite{BCS14}. We refer the reader to this paper for further details. Finally, in Section \ref{sec:pf} we prove our main Theorems \ref{thm:main} and \ref{thm:main2}. 

\subsection*{Acknowledgements.} 
The first author was partially supported by the DAAD during his visit to the Gottfried Wilhelm Leibniz Universit\"at Hannover in the Academic Year 2010/2011, when this research project started.
The second author gratefully acknowledges the support by the
Institut f\"ur Analysis, Fakult\"at f\"ur Mathematik und Physik, Gottfried Wilhelm Leibniz Universit\"at Hannover, during his stays as Visiting Scientist in the Academic Years 2011/2012 and 2012/2013,
when this research has been partly developed.

\section{Generating Functions for a Class of\\Boundary-preserving Symplectomorphisms}
\label{sec2}

The following lemma, which is proven in \cite{ME81}, analyzes symplectomorphisms of the type we consider. 
\begin{Lem}
\label{lem:chidelta}
Let $X$, $Y$  and $\chi$ be as above. 
Then $\chi$ induces a symplectomorphism $\chi_{\partial}: T^*\partial Y \setminus 0 \to T^*\partial X\setminus 0$,
 positively homogeneous of order one in the fibers, such that the following diagram commutes:
\[
\xymatrix{
&T^*_{\partial Y}Y \setminus N^*{\partial Y} \ar@{^(->}[d]^{i^*_Y}  \ar[r]^{\chi}  & T^*_{\partial X}X
\setminus N^*{\partial X}  \ar@{^(->}[d]^{i^*_X} \\
&T^*\partial Y \setminus 0 \ar[r]^{\chi_{\partial}}  & T^*\partial X \setminus 0.
}
\]
\end{Lem}

\begin{rem}
\label{rem:triv}
In Lemma \ref{lem:chidelta} we have considered the induced symplectomorphism $\chi_{\partial}$ outside the zero section. Actually, since $\chi$ is smooth on $\partial T^*Y \setminus 0$, the induced symplectomorphism $\chi_{\partial}$
is also smooth on the zero section. Since $\chi_{\partial}$ is positively homogeneous of order one in the fibers, the smoothness at the zero section implies that $\chi_{\partial}$ is trivial in the fibers. 
That is, $\chi_{\partial}$ is the lift of a diffeomorphism of the boundaries, cf.\ \cite{DAS01}. 
\end{rem}

It is useful to study the Jacobian of the local representation of $\chi$ in a collar neighborhood of the boundaries. We write
\begin{equation}
\label{eq:compsympl}
 \begin{split}
  \chi: T^*Y \setminus 0 &\to T^*X\setminus 0\\
  (y', y_n, \eta', \eta_n)& \\
 \mapsto&(x'(y', y_n, \eta', \eta_n), x_n(y', y_n, \eta', \eta_n), \xi'(y', y_n, \eta', \eta_n), \xi_n(y', y_n, \eta', \eta_n)),\nonumber
 \end{split}
\end{equation}
where the coordinates $(y', y_n, \eta', \eta_n)$, $(x', x_n, \xi', \xi_n)$ determine a collar neighborhood of the boundary, that is, $y_n, x_n$ are (local) boundary defining functions on $Y$ and $X$, respectively. 
Since the boundary is preserved, $x_n(y', 0, \eta', \eta_n)=0$ for all $(y', \eta', \eta_n)$, hence $\partial_{y'}x_n$, 
$\partial_{\eta'}x_n$, $\partial_{\eta_n} x_n$ are identically zero at $y_n=0$. Moreover,  Lemma \ref{lem:chidelta}
implies that $x'$ and $\xi'$ define a symplectomorphism on the cotangent bundle of the boundary which is independent
of the conormal direction, that is $\partial_{\eta_n}x'$ and $\partial_{\eta_n}\xi'$ are identically zero at the boundary. 
Hence, the Jacobian of $\chi$ at the boundary has the form
\begin{equation}
\label{eq:jacob}
 J(\chi)|_{y_n=0}=\left(
\begin{array}{cclc}
\partial_{y'}x'_\partial & \partial_{\eta'}x'_\partial & \partial_{y_n}x'|_{y_n=0}& 0 \\
\partial_{y'}\xi'_\partial & \partial_{\eta'}\xi'_\partial & \partial_{y_n}\xi'|_{y_n=0}& 0 \\
0 & 0 & \partial_{y_n}x_n'|_{y_n=0}& 0 \\
\partial_{y'}\xi_n'|_{y_n=0} & \partial_{\eta'}\xi_n'|_{y_n=0} & \partial_{y_n}\xi_n'|_{y_n=0}& \partial_{\eta_n}\xi_n'|_{y_n=0} 
\end{array} \right),
\end{equation}
where $x'_\partial, \xi'_\partial$ are the functions $x',\xi'$ evaluated at $y_n=0$. From Lemma \ref{lem:chidelta} we know that $\chi$ induces a symplectomorphism $\chi_\partial$ on the boundary. Therefore
\begin{equation}
 \label{eq:simplbound}
 J(\chi_\partial)=
 \left(
 \begin{array}{cc}
  \partial_{y'}x'_\partial & \partial_{\eta'}x'_\partial\\
\partial_{y'}\xi'_\partial & \partial_{\eta'}\xi'_\partial 
 \end{array}
 \right)
\end{equation}
is a symplectic matrix, hence it has determinant $1$, with $x'_\partial, \xi'_\partial$ interpreted as the components of 
$\chi_\partial$ in the local coordinates $(y^\prime,\eta^\prime)$ on the boundary. Clearly, also $J(\chi)|_{y_n=0}$ has determinant equal to $1$, since $\chi$ is a symplectomorphism, and this implies that $\partial_{y_n}x_n \times \partial_{\eta_n}\xi_n=1$ for $y_n=0$. In particular $\partial_{y_n}x_n, \partial_{\eta_n}\xi_n$ can never vanish at the boundary. Since the boundary is compact,
these two functions are actually bounded away from zero when $y_n=0$, and therefore also in a sufficiently small collar neighborhood of the boundary.

We now recall a well known property of Lagrangian subspaces, which can be extended to the case of manifolds with boundary. We denote by $Z$ a manifold without boundary.
\begin{Prop}\label{prop:loclag}
Let $\Lambda \subset T^*Z\setminus 0 $ be a conic Lagrangian submanifold. Then, for all $\lambda_0=(z_0,\eta_0) 
\in \Lambda$, there exist a neighborhood $U_{z_0}$ and a phase function $\phi$ defined in a conic neighborhood
$U_{z_0} \times \Gamma$ in $U_{z_0} \times \R^N$, $N$  large enough, such that $\phi$ parametrizes $\Lambda$ in a conic neighborhood of $\lambda$. That is,
\[
  C_{\phi}=\{(z, \theta) \mid \phi'_\theta(z, \theta)=0\} \to T^*Z\setminus0 \colon(z, \theta) \mapsto (z, \phi'_z(z, \theta))
\]
induces a diffeomorphism in a small conic neighborhood  $U^\Lambda_{\lambda_0}$ of $\lambda_0$ in $\Lambda$.

Moreover, if $\tilde\Lambda \subseteq (T^*\tilde Y\setminus 0) \times (T^*\tilde X\setminus0)$ 
is locally defined by the graph of a symplectomorphism
$$\tilde \chi: T^*\tilde Y\setminus 0 \to T^*\tilde X\setminus0,$$
the phase function can be written in the form 
$$\phi(x, y, \theta)=\psi(x, \theta)- \langle y , \theta\rangle,$$ 
with $\phi \in C^\infty(\Omega_{x_0} \times \Omega_{y_0} \times
\Gamma)$, with $\Omega_{x_0}$ and $\Omega_{y_0}$ neighborhoods of $x_0\in \tilde X$ and $y_0\in \tilde Y$,  respectively,
and $\Gamma$ a cone in $\R^n\setminus 0$, $2n$ being the dimension of $\Lambda$.
\end{Prop}
\begin{rem}
\label{rem:psipartial}
 In Remark \ref{rem:triv}, we  noticed that $\chi$ induces a symplectomorphism 
 $\chi_\partial: T^*\partial Y \setminus 0 \to T^*\partial X\setminus 0$ 
 which is again positively homogeneous in the fibers. 
 Applying Proposition \ref{prop:loclag} to $\chi_\partial$ we obtain a phase function 
 $\phi_\partial(x', y', \theta')=\psi_\partial (x', \theta') - \langle y', \eta'\rangle$ 
 which represents $\chi_\partial$. Since $\chi_\partial$ is the lift of a diffeomorphism, the phase function $\psi_\partial (x',\theta')$
 is smooth at $\theta'=0$, therefore it is linear.
\end{rem}

For the sake of brevity, we do not recall the notion of Maslov bundle. For its description, see, e.g., \cite{HO03}.
The proof of the next Lemma \ref{lem:mas} can be found in \cite{BCS14}.

\begin{Lem}
\label{lem:mas}
The Maslov bundle of 
\[
\Lambda=\textnormal{graph}(\chi)'= \{(x, \xi), (y, -\eta)\mid \chi(y, -\eta)=(x, \xi)\} 
\subseteq (T^*X\setminus 0) \times (T^*Y \setminus 0)
\]
is trivial in a neighborhood of $\partial \Lambda= (\partial T^*Y \times \partial T^*X) \cap \Lambda$.
\end{Lem}

In order to define a suitable calculus for Fourier integral operators on manifolds with boundary, we need to introduce the transmission condition, see, e.g., \cite{BU71,GR87,GH90,RS85,SC01}. Consider 
the function spaces:
\[
H^+= \{\mathscr{F}(e^+u)\mid u \in  \Si(\R_+)\} \quad \mbox{ and } \quad H^-_0=\{\mathscr{F}(e^-u)\mid u \in \Si(\R_-)\},
\]
where $\Si(\R_\pm)= r^\pm \Si(\R)$ is the restriction of the Schwartz  functions on $\R$ to the right (left) half line,
$e^\pm$ is the extension by zero to $\R$ of a function defined on $\R_\pm$.
It is easy to prove that $H^+$ and $H^-_0$ are spaces of functions decaying of first order at infinity.
 Moreover, we denote by  $H'$ the set of all polynomials in one variable. Then we  define
\[
H= H^+ \oplus H^-_0 \oplus H'.
\]

\begin{Def}
\label{Def:trans}
Let $a \in S^m(\R^n \times \R^n \times \R^n)$. Then
$a$ satisfies  the transmission condition  at $x_n=y_n=0$ when, for all $k,l$,
\[
\partial^k_{y_n} \partial_{x_n}^l a(x', 0, y', 0, \xi', \langle \xi' \rangle \xi_n) \in S^m(\R^{n-1} \times 
\R^{n-1} \times \R^{n-1}) \hat{\otimes}_{\pi} H_{\xi_n}.
\]
We denote by $S^m_{\tr}(\R^{n} \times \R^n \times \R^n)$ the subset of symbols of order $m$ satisfying the transmission condition.
\end{Def}
\noindent
For symbols positively homogeneous of order $m$ with respect to  the $\xi$ variable, Definition \ref{Def:trans}
is equivalent to 
\begin{equation}
\label{eq:hom}
\partial_{x_n}^k \partial_{y_n}^l\partial_{\xi'}^\alpha \partial_{x'}^\beta a(x', 0, y', 0, 0,1)=
 (-1)^{m-|\alpha|} \partial_{x_n}^k \partial_{y_n}^l\partial_{\xi'}^\alpha \partial_{x'}^\beta a(x', 0, y', 0,0, -1)
\end{equation}
for all $k, l\in \N$, $\alpha, \beta  \in \N^{n-1}$. The above condition is often called symmetry condition;
the proof of the equivalence can be found, e.g., in \cite[Section 18.2]{HO03}. 
\begin{Def}[Admissible symplectomorphisms]\label{def:adm}
We say that $\chi$ is admissible, if all its components locally satisfy the transmission condition at the boundary. A phase function $\psi$ representing an admissible symplectomorphism will be called admissible. 
The first derivatives of $\psi$ are then homogeneous symbols, which satisfy the transmission condition.
\end{Def}

\begin{rem}
$($a$)$  Definition \ref{def:adm} has an invariant meaning: A change of coordinates in the cotangent bundle, induced by a change of coordinates in the base manifold, is linear with respect to the fibers.
Hence, if the transmission condition is satisfied in a local chart, it is also fulfilled after a change of coordinates. 

$($b$)$ When we refer to the phase function $\psi$ as a symbol then this is only correct 
after a modification near $|\xi|=0$. Modulo operators with smooth kernel, the precise form of this modification is irrelevant. 
\end{rem}

Lemma \ref{lem:chidelta} and Definition \ref{def:adm} imply some useful properties of the phase function $\psi$. First, $\psi^\prime_{\eta_n}(x', 0, \eta', \eta_n)$ is identically equal to zero for $\eta'\not=0$, hence 
$\psi(x', 0, \eta', \eta_n)$ does not depend on $\eta_n$ for $\eta'\not=0$. We set
\begin{equation}\label{eq:psidelta}
	\psi_\partial(x', \eta')=\psi(x', 0, \eta', \eta_n).  
\end{equation}
Note that $\psi_\partial$ is a generating function of the symplectomorphism $\chi_\partial\colon T^*\partial Y\setminus0\to T^*\partial X\setminus 0$ described in Remark \ref{rem:triv}. A further consequence of $\psi$ being a local phase function associated with a symplectomorphism which preserves the boundary is that $\psi_\partial(x^\prime,\eta^\prime)$ is linear in $\eta^\prime$, so that, in particular, $\psi(x^\prime,0,0,\eta_n)\equiv0$. Moreover, since the phase function is regular up to boundary
and $\xi_n=\psi^\prime_{x_n}(x,\eta)$, by \eqref{eq:jacob} and the subsequent considerations,
\begin{equation}\label{eq:nondeg}
	\partial_{\eta_n}\xi_n|_{y_n=0}=
	\psi^{\prime\prime}_{x_n\eta_n}(x^\prime(y^\prime,0,\eta^\prime,\eta_n),0,\eta^\prime,\eta_n)\not=0
	\Rightarrow \psi^{\prime\prime}_{x_n\eta_n}(x^\prime,0,\eta^\prime,\eta_n)\not=0.
\end{equation}
By continuity and compactness, the property remains true in a sufficiently small collar neighborhood of the boundary.

\section{A Class of Fourier Integral Operators\\on Manifolds with Boundary}
\label{sec3}
In this section we introduce the Fourier integral operators we are interested in and
describe their mapping properties, cf. \eqref{eq:map}.
%
Consider $A^\chi \in I^m_{\mathrm{comp}} (\widetilde{Y}, \widetilde{X}, \widetilde{\Lambda})$, where $\widetilde{\Lambda}= \textnormal{graph}(\tilde{\chi})'$. 
The definition implies that for all $(y_0, x_0, \eta_0, \xi_0)=\lambda_0 \in \widetilde{\Lambda}$, 
$A^\chi$ is, microlocally, a linear operator associated with a kernel which in local coordinates $(x,y)$ in the open set $\Omega_x \times \Omega_y$ parametrizing a 
neighborhood of $(x_0,y_0)$ is of the form
\[
k_{A^\chi}(x,y)=\int e^{i \phi(x',x_n,y', y_n, \xi',\xi_n )}a(x', x_n, y', y_n,\xi', \xi_n) \dbar \xi' \dbar \xi_n.
\]
We write $x=(x', x_n)$, $y=(y', y_n)$ with boundary defining functions  $x_n$ and $y_n$, respectively.
The phase function $\phi$ is defined in $\Omega_x \times \Omega_y \times \Gamma$ with $\Gamma $ open and conic in $\R^n \setminus \{0\}$, and the symbol $a(x, y,\xi)$ has support contained in $\Omega_x \times \Omega_y \times\Gamma\subset\R^{2n} \times \R^n$.
Proposition \ref{prop:loclag} implies that we can choose
\[
	\phi(x', x_n, y', y_n, \xi', \xi_n)= \psi(x', x_n, \xi', \xi_n)- \langle y', \xi'\rangle- 
	\langle y_n, \xi_n\rangle. 
\]

It is natural also to assume that $a$ has the transmission property with respect to $x_n=0$, $y_n=0$. In fact, this is necessary  to ensure that we obtain a continuous linear mapping $A^\chi\colon C^\infty(Y)\to C^\infty(X)$ even in the simpler case of pseudodifferential operators on $X$ (i.e., $Y=X$ and $\tilde{\chi}=id$). Note that, when computing the derivatives of $A^\chi u$, $u\in C^\infty(Y)$, close to $\partial X$, one needs to study expressions 
\[
	\iint e^{i \phi(x,y,\xi)}c(x,y,\xi) u(y) \dbar \xi dy,
\] 
where $c$ belongs to the span of symbols of the form
\begin{equation}\label{eq:amplder}
	\left[\prod_j (\partial^{\beta_j}_x\psi)(x,\xi) \right] \cdot (\partial^\gamma_x a)(x,y,\xi). 
\end{equation}
As a consequence of the assumptions on $\chi$ and $a$, the symbols of this type 
have the transmission property and 
$A^\chi$ maps $C^\infty(Y)$ continuously to 
$C^\infty(X)$, see \cite{BCS14}. 
The necessity of the conditions on $\chi$ and $a$ for this continuity property will be discussed elsewhere.

For simplicity, we will consider in the sequel $\psi$ defined on $\R^n \times \R^n \setminus \{0\}$.
The extension we choose is not relevant, since the symbol $a$ vanishes outside 
$\Omega_x\times \Omega_y\times \Gamma$. 
We can then focus on operators with kernel given by oscillatory integrals 
\begin{equation}
\label{eq:osciint}
\int e^{i(\psi(x', x_n, \xi', \xi_n) - y'\cdot \xi'- y_n \cdot \xi_n )}
 a(x',x_n, y', y_n,\xi', \xi_n) \dbar \xi_n \dbar\xi_n,
\end{equation}
with $\psi$ the generating function of an admissible symplectomorphism and $a$ a symbol of order $m$ with the transmission property. 
We stress the fact that the operator $A^\chi$ admits both a right and a left quantization, since the phase function represents locally a symplectomorphism. In particular, the symbol $a$ appearing in \eqref{eq:osciint} can be chosen independent of $y$, modulo a smoothing operator, see \cite{HO04}, Chapter 25.

For $u\in C^\infty(Y)$, supported close to $\partial Y$, $A^\chi u$ is then given, close to the boundary $\partial X$, by a finite sum of microlocal terms of the form
\begin{align*}
	\int &e^{i\psi(x,\xi)}a(x,\xi)\widehat{u}(\xi)\dbar\xi
	\\
	&=
	\int e^{i\psi_\partial(x',\xi')}\left[\int e^{i(\psi(x',x_n,\xi',\xi_n)-\psi_\partial(x',\xi'))}a(x',x_n,\xi',\xi_n)
	\widehat{u}(\xi',\xi_n)\dbar\xi_n\right]\dbar\xi',
\end{align*}
modulo operators with smoothing kernel. This holds in view of the fact that the Maslov bundle is here trivial and all phase functions are equivalent. 
Using for example the results of \cite{LSV94}, it is possible to find a global phase function in a small neighborhood of $x_n=0$, so the sum can actually be reduced to a single term. We can then interpret the operator $A^\chi$ as an \textit{operator-valued Fourier integral operator}, locally defined on the half-space $\R^n_+$, in analogy with \cite{SC01,SC98}, and focus on the operators of the form \eqref{eq:opan}, that is, on the action in the normal direction $A_n^\chi$.

This allowed us to introduce in \cite{BCS14} the class of \textit{Fourier integral operators of Boutet de Monvel type}, of which we now shortly recall the local definition. 
For $\s=(s_1,s_2)\in \R^2$ we introduce the weighted Sobolev spaces 
%
\[
	H^\s (\R^n) = H^{s_1,s_2}(\R^n)=\norm{x}^{s_2}H^{s_1}(\R^n)
\]
with the usual (unweighted) Sobolev space $H^{s_1}(\R^n)$. 
The corresponding spaces on $\R^n_+$ are obtained by restriction, and endowed with the natural topology. We set 
\[
 \partial_+= r^+ \partial_{x_n}e^+: H^{s_1, s_2}(\R_+) \to H^{s_1-1, s_2}(\R_+), \quad s_1> -\frac{1}{2}.
\]
One can consider the operator $\partial_+$ as an operator-valued symbol belonging to
$S^{1}(\R^{n-1}, \R^{n-1}; H^{\s}(\R_+), H^{\s-(1,0)}(\R_+))$. Let us now recall the definition of
local \emph{potential symbols}, \emph{trace symbols}, \emph{singular Green symbols}:
\begin{itemize}
 \item [i)] A \emph{potential symbol} of order $m$ is an element of
\[
 S^{m}(\R^{n-1}, \R^{n-1}; \C, \Si(\R_+))= \proj_\mathbf{s}S^m(\R^{n-1}, \R^{n-1}; \C, H^\mathbf{s}(\R_+)).
\]
\item[ii)] A \emph{trace symbol} of order $m$ and type zero is an element of the set
\[
 S^m(\R^{n-1}, \R^{n-1}; \Si'(\R_+), \C)=\proj_\mathbf{s}S^m(\R^{n-1}, \R^{n-1}; H^\mathbf{s}_0
 (\overline{\R}_+), \C).
\]
Clearly, a trace symbol of order $m$ and type zero defines also a symbol 
in $S^{m}(\R^{n-1}, \R^{n-1};H^\mathbf{s}(\R_+), \C)$, if $s_1>-\frac{1}{2}$. 
A \emph{trace symbol} of type $d$ is a sum of the form
\[
 t=\sum_{j=0}^d t_j \partial_+^j, \quad t_j \in S^{m-j}(\R^{n-1}, \R^{n-1}; \Si'(\R_+), \C).
\]
Then $t$ is in $S^{m}(\R^{n-1}, \R^{n-1}; H^\mathbf{s}(\R_+), \C)$ for $s_1> d- \frac{1}{2}$.
\item[iii)] A \emph{singular Green symbol} of order $m$ and type zero is an element of
\[
S^m(\R^{n-1}, \R^{n-1}; \Si'(\R), \Si(\R_+))= 
\limproj_\mathbf{s} S^m(\R^{n-1}, \R^{n-1}; H_0^\mathbf{s}
(\overline{\R}_+),H^\mathbf{s}(\R_+) ).
\]
A singular Green symbol of order $m$ and type zero gives a symbol in $S^{m}(\R^{n-1},$
$\R^{n-1}; H^\mathbf{s}(\R_+), \Si(\R_+))$, provided $s_1>-\frac{1}{2}$. A \emph{singular Green symbol} of  order $m$ and type $d$ is a sum of the  form
\[
 g=\sum_{j=0}^d g_j \partial_+^j, \quad g_j \in S^{m-j}(\R^{n-1}, \R^{n-1}; \Si'(\R_+), \Si(\R_+)).
\]
We find that $g$ is in $S^m(\R^{n-1}, \R^{n-1}; H^\mathbf{s}(\R_+), \Si(\R_+))$ for $s_1> d- \frac{1}{2}$.
\end{itemize}

\begin{rem}
\label{rem:dirac}
It is well known that the trace operator $\gamma_j$, given by $(\gamma_ju)(y')=(\partial_{y_n}^ju)(y',0)$, is a trace symbol of order $j+\frac{1}{2}$ and type $j+1$,
see \cite{SC01}.
\end{rem}

With the notation introduced above we can now define the relevant operator class:
\begin{Def} 
\label{def:Fourier integral operatorBdM}
We denote by  $\mathscr{B}_\chi^{m,d}(X, Y)$ the class of all operators 
\[
 \mathcal{A}:=\left(
\begin{array}{cc}
r^+ A^\chi e^+ + G^{{\chi}_\partial} & K^{\chi_\partial}\\
 T^{\chi_{\partial}}& S^{\chi_\partial}
\end{array} \right) 
: \begin{array}{c}
 C^\infty(Y)\\
\oplus\\
 C^\infty(\partial Y)
\end{array}
\to 
\begin{array}{c}
 C^\infty(X)\\
\oplus\\
 C^\infty(\partial X)
\end{array} .
\]
Here $A^\chi \in I^m_{\mathrm{comp}}(\widetilde{X}, \widetilde{Y}, \widetilde{\Lambda})$
is as defined above. 
Modulo operators with smooth kernel in the interior, the other entries are described as follows: 
$G^{\chi_\partial}$ is a Fourier integral operator with Lagrangian submanifold defined by 
$\mathrm{graph}(\chi_{\partial})'$ and local singular Green symbol $g$
 of order $m$ and type $d$;
$K^{\chi_\partial}$ is a Fourier integral operator with Lagrangian submanifold defined by 
$\mathrm{graph}(\chi_{\partial})'$ and local potential symbol $k$ of  order $m$;
$T^{\chi_\partial}$ is a Fourier integral operator with Lagrangian submanifold  defined by 
$\mathrm{graph}(\chi_{\partial})'$ and local trace symbol $t$ of  order $m$ and type $d$;
$S^{\chi_\partial}$ is a Fourier integral operator with Lagrangian submanifold  defined by 
$\mathrm{graph}(\chi_{\partial})'$ and local symbol $s \in S^{m}(\R^{n-1}, \R^{n-1})$. 
\end{Def}

\section{Proof of Theorems \ref{thm:main} and \ref{thm:main2}}\label{sec:pf}

We now prove the main results of the paper, stated in the Introduction.

\begin{proof}[Proof of Theorem \ref{thm:main}]
Both $\psi$ and $\psi_\partial$ are H\"ormander symbols belonging to $S^1_{1,0}(\Omega\times\R^n)$ and $S^1_{1,0}(\Omega^\prime\times\R^{n-1})$, respectively.
positively homogeneous of degree one in the covariable (outside the zero-section). Moreover, since $\psi$ is a (local) phase function associated with a
symplectomorphism which preserves the boundary of the underlying manifolds, we have $\psi_{\xi_n}^\prime(x^\prime,0,\xi^\prime,\xi_n)=0$
for any $(x^\prime,\xi^\prime,\xi_n)$. In particular, then, as recalled in Section \ref{sec2}, $\psi(x^\prime,0,\xi^\prime,\xi_n)=\psi_\partial(x^\prime,\xi^\prime)$ 
does not depend on $\xi_n$. This implies, additionally, that not only $\varphi(x^\prime,0,\xi^\prime,\xi_n)=0$ for any $(x^\prime,\xi^\prime,\xi_n)$,
but also that $\partial^\alpha_{\xi_n}\varphi(x^\prime,0,\xi^\prime,\xi_n)=\partial^\alpha_{\xi_n}\psi(x^\prime,0,\xi^\prime,\xi_n)=0$ for any $\alpha\in\Z_+\setminus\{0\}$ and any $(x^\prime,\xi^\prime,\xi_n)$.
As a consequence, for any $\alpha\in\Z_+$ and any $(t,\tau)\in\R^2$, $t\not=0$, 
$(x^\prime,\xi^\prime)\in \Omega^\prime\times(\R^{n-1}\setminus\{0\})$,
there exists $\theta$ between $0$ and $t$ such that 
%
\begin{align*}
	\omega\left(\frac{t}{\norm{\xi^\prime}}\right)
	&\left|\partial^\alpha_{\tau}
	\left[\varphi\left(x^\prime,\frac{t}{\norm{\xi^\prime}},\xi^\prime,\tau\norm{\xi^\prime}\right)\right]\right|
	\\
	=&\;\omega\left(\frac{t}{\norm{\xi^\prime}}\right)
	\norm{\xi^\prime}^{\alpha}
	\left|(\partial^\alpha_{\xi_n}\varphi)\left(x^\prime,\frac{t}{\norm{\xi^\prime}},\xi^\prime,\tau\norm{\xi^\prime}\right)
	-
	(\partial^\alpha_{\xi_n}\varphi)\left(x^\prime,0,\xi^\prime,\tau\norm{\xi^\prime}\right)
	\right|
	\\
	=&\;\omega\left(\frac{t}{\norm{\xi^\prime}}\right)\dfrac{|t|}{\norm{\xi^\prime}}\cdot\norm{\xi^\prime}^{\alpha}
	\left|(\partial_{x_n}\partial^\alpha_{\xi_n}\varphi)\left(x^\prime,\frac{\theta}{\norm{\xi^\prime}},\xi^\prime,\tau\norm{\xi^\prime}\right)\right|
	\\
	\lesssim&\; \norm{t}\norm{\xi^\prime}^{\alpha-1}\norm{(\xi^\prime,\tau\norm{\xi^\prime})}^{1-\alpha}
	\le\norm{t}\norm{\tau}^{1-\alpha},
\end{align*}
that is, for any $(t,\tau)\in\R^2$,

\vspace{-1mm}

\begin{equation}
	\label{eq:2}
	\omega\left(\frac{t}{\norm{\xi^\prime}}\right)\left|\partial^\alpha_{\tau}
	\left[\varphi\left(x^\prime,\frac{t}{\norm{\xi^\prime}},\xi^\prime,\tau\norm{\xi^\prime}\right)\right]\right|
	\lesssim\norm{t}\norm{\tau}^{1-\alpha},
\end{equation}

\vspace{-1mm}

\noindent
with constants independent of $(x^\prime,\xi^\prime)\in U^\prime\times(\R^n\setminus\{0\})$, $U^\prime\subset\subset\Omega^\prime$.

Remember that, in view of the hypotheses and the properties of $\psi$ deduced in Section \ref{sec2},
$\psi(x^\prime,0,\xi^\prime,\xi_n)=\psi_\partial(x^\prime,\xi^\prime)$ is linear in $\xi^\prime$, so that, in particular, $\psi(x^\prime,0,0,\xi_n)\equiv0$.
Due to homogeneity, we have near the boundary
\begin{align*}
	\psi(&x^\prime,x_n,0,\xi_n)
	\\
	&=x_n\psi^\prime_{x_n}(x^\prime,0,0,\xi_n)+x_n^2\int_0^1(1-s)\psi^{\prime\prime}_{x_n x_n}(x^\prime,sx_n,0,\xi_n)ds
	\\
	&=\left[\pm x_n\psi^\prime_{x_n}(x^\prime,0,0,\pm1)\pm x_n^2\int_0^1(1-s)\psi^{\prime\prime}_{x_n x_n}(x^\prime,sx_n,0,\pm1)ds\right]\xi_n,
	\; \xi_n\gtrless 0.
\end{align*}
Differentiating with respect to $x_n$ shows that 
\begin{equation}
	\label{eq:5}
	\psi^\prime_{x_n}\left(x^\prime,x_n,0,\xi_n\right)
	=
	[q^\pm(x^\prime)+ x_n r^\pm(x)]\xi_n,\;\xi_n\gtrless0,
\end{equation}
with $q^\pm\in C^\infty(\Omega^\prime)$, $r^\pm\in C^\infty(\Omega)$. 
As $\psi$ satisfies the transmission condition, \eqref{eq:hom} implies that 
$$q^+(x') = -q^-(x').$$
More is true: Since $\psi$ is a regular phase function,
we know from \eqref{eq:nondeg} that $\psi^{\prime\prime}_{x_n\xi_n}(x,\xi)\not=0$ everywhere on $\Omega\times(\R^n\setminus\{0\})$.
Without loss of generality we can assume it to be positive 
everywhere on $\Omega\times(\R^n\setminus\{0\})$.
As a consequence
$$q^\pm(x^\prime)=\pm\psi^{\prime\prime}_{x_n\xi_n}(x^\prime,0,0,\xi_n)\gtrless0, \xi_n\gtrless0, x^\prime\in\Omega^\prime.$$
Let $U'\subset\subset \Omega'$. 
Then there exists a $\kappa>0$ such that 
$$|q^\pm(x^\prime)|\ge4\kappa\text{ for }x^\prime\in U^\prime.$$ 
By continuity and the compactness of $U^\prime$, there exist $k>0$ and $\rho>0$ sufficiently small such that 
\[
	\pm\psi^\prime_{x_n}\left(x^\prime,x_n,\xi^\prime,\pm1\right)\ge\kappa>0 \text{ for }
	x^\prime\in U^\prime, |x_n|\le k, |\xi^\prime|\le\rho.
\]
For convenience we will assume in the sequel that we can take $k=1=\rho$. 
For $(t,\tau)\in\R^2$ we then obtain
\begin{equation}
	\label{eq:6}
	\begin{aligned}
	\psi^\prime_{x_n}\left(x^\prime,\frac{t}{\norm{\xi^\prime}},\frac{\xi^\prime}{\norm{\xi^\prime}},\tau\right)=
	\pm\psi^\prime_{x_n}\left(x^\prime,\frac{t}{\norm{\xi^\prime}},\frac{\xi^\prime}{|\tau|\norm{\xi^\prime}},\pm1\right)\tau\ge\kappa|\tau|>0
	\end{aligned}
\end{equation}
for $x^\prime\,\in U^\prime, \xi^\prime\in\R^{n-1}\setminus\{0\}, 
\frac{|t|}{\norm{\xi^\prime}}\le 1, |\tau|\ge1$. 
\smallskip
We now fix the cut-off function $\omega$, define $\sphi$ as in \eqref{eq:1} and check conditions (P1), (P2),  and (P3) on a regular SG phase function. 

Ad (P1). For any choice of $a,\alpha\in\Z_+$, $\partial^a_t\partial^\alpha_\tau \sphi(t,\tau)$ 
is a linear combination of terms of the form 
\begin{eqnarray*}
S_j &=&\omega^{(j)}\left(\frac{t}{\norm{\xi^\prime}}\right)\norm{\xi^\prime}^{\alpha-a}\cdot
	(\partial^{a-j}_{x_n}\partial^{\alpha}_{\xi_n}\varphi)\left(x^\prime,\frac{t}{\norm{\xi^\prime}},\xi^\prime,\tau\norm{\xi^\prime}\right),\text{ and}\\
T_j&=&K\left[1-\omega\left(\frac{t}{\norm{\xi^\prime}}\right)\right]^{(j)}_t\cdot\partial^{a-j}_t\partial^\alpha_\tau(t\cdot\tau), 
\end{eqnarray*}
$\ j=0,\ldots, a.$
%
%
The summands $S_j$, $j=0,\dots,a$, can be estimated as follows.
\begin{align}
				|S_j|&=\left|\omega^{(j)}\!\!\left(\frac{t}{\norm{\xi^\prime}}\right)\right|\norm{\xi^\prime}^{\alpha-a}
				\cdot\left|(\partial^{a-j}_{x_n}\partial^{\alpha}_{\xi_n}\varphi)\left(x^\prime,\frac{t}{\norm{\xi^\prime}},\xi^\prime,\tau\norm{\xi^\prime}\right)\right|
				\nonumber
				\\
				&\lesssim
				\left|\omega^{(j)}\!\!\left(\frac{t}{\norm{\xi^\prime}}\right)\right| 
				\norm{\xi^\prime}^{\alpha-a}\cdot\norm{(\xi^\prime,\tau\norm{\xi^\prime})}^{1-\alpha}\nonumber \\
				&=
				\left|\omega^{(j)}\!\!\left(\frac{t}{\norm{\xi^\prime}}\right)\right| 
				\norm{\xi^\prime}^{1-a}\cdot\norm{\tau}^{1-\alpha}.
				\label{eq:sj}
\end{align}
\begin{itemize}
\item For $j>0$ we note that on $\mathrm{supp}\, \omega^{(j)}(t/\norm{\xi'})$ we have 
$\norm{\xi'}\sim \norm{t}$, so that we  can estimate $S_j$ by 
$\norm{t}^{1-a}\norm{\tau}^{1-\alpha}$, as asserted.   
\item For $j=0$ and $a=0$, the required estimate is given by \eqref{eq:2}.
\item For $j=0$ and $a=1$, the estimate is trivial from \eqref{eq:sj}.
\item For $j=0$ and $a>1$ it suffices to show that 
$$\omega(t/\norm{\xi'}) \lesssim \frac{\norm{\xi'}^{1-a}}{\norm{t}^{1-a}} .
$$ 
This is clear for $\norm{\xi'}\ge\norm t$. As $\omega(t)$ vanishes for $|t|\ge1$, 
it remains to consider the case, where 
$$|t|\le \norm{\xi'}\le \norm t.$$
But then $\norm t \le 1+|t| \le 2\norm{\xi'}\le 2 \norm t,$ and the assertion also follows.
\end{itemize}
The summands $T_j$, $j=0,\dots,a$, are given by linear combinations of terms of 
the form 
$$ K \ \partial_t^j\left(1-\omega\left(\frac t{\norm{\xi'}}\right)\right)\partial _t^{a-j}t\, \partial^\alpha_\tau \tau.$$
So it remains to check that 
$$ \partial_t^j\left(1-\omega\left(\frac t{\norm{\xi'}}\right)\right) \partial _t^{a-j}t 
\lesssim \norm t^{1-a}.
$$
\begin{itemize}
\item For $j=0$ there is nothing to show. 
\item For $j>0$, we note again that on $\mathrm{supp}\, \omega^{(j)}(t/\norm{\xi'})$ we have $\norm{\xi'}\sim \norm{t}$, so that we obtain the desired estimate.
\end{itemize}	%
%

\noindent
Ad (P2). The estimates from above, with constant $C>0$ independent of $(x^\prime,\xi^\prime)\in U^\prime\times(\R^{n-1}\setminus\{0\})$, are special cases of the considerations for (P1) with $a=1, \alpha=0$ and $a=0,\alpha=1$, respectively.
Let us prove that the two estimates from below hold, provided we choose $K$ large enough.
\begin{enumerate}
	\item[ i)] From the homogeneity and the properties of $\varphi$ and $\psi$ explained above, we obtain
	\begin{align*}
		\sphi^\prime_\tau(t,\tau)&=
		\omega\left(\frac{t}{\norm{\xi^\prime}}\right)\cdot
					\varphi^\prime_{\xi_n}\left(x^\prime,\frac{t}{\norm{\xi^\prime}},\xi^\prime,\tau\norm{\xi^\prime}\right)
					+\left[1-\omega\left(\frac{t}{\norm{\xi^\prime}}\right)\right]K\cdot t
		\\
		&=
		\omega\left(\frac{t}{\norm{\xi^\prime}}\right)\cdot
					\psi^\prime_{\xi_n}\left(x^\prime,\frac{t}{\norm{\xi^\prime}},\frac{\xi^\prime}{\norm{\xi^\prime}},\tau\right)\norm{\xi^\prime}
					+\left[1-\omega\left(\frac{t}{\norm{\xi^\prime}}\right)\right]K\cdot t.
	\end{align*}
As $\psi'_{\xi_n}(x',0,\xi)=0$ for all $\xi\not=0$ this term vanishes for $t=0$; for $t\not=0$
we find $\theta$ between $0$  and  $t/\norm{\xi^\prime}$ such that 
\begin{align*}
		\sphi^\prime_\tau(t,\tau)&=
		\displaystyle
		\left\{
		\omega\left(\frac{t}{\norm{\xi^\prime}}\right)\cdot
					\psi^{\prime\prime}_{x_n\xi_n}\left(x^\prime,\theta,\frac{\xi^\prime}{\norm{\xi^\prime}},\tau\right)
					+\left[1-\omega\left(\frac{t}{\norm{\xi^\prime}}\right)\right]K\right\}\cdot t
\end{align*}
	%
	We will now study the coefficient of $t$,
	\[
		A=\omega\left(\frac{t}{\norm{\xi^\prime}}\right)\cdot
					\psi^{\prime\prime}_{x_n\xi_n}\left(x^\prime,\theta,\frac{\xi^\prime}{\norm{\xi^\prime}},\tau\right)
					+\left[1-\omega\left(\frac{t}{\norm{\xi^\prime}}\right)\right]K
	\]
and show that $A\ge c_1$ for some $c_1>0$. For this, however, it is sufficient to note that $A$ is a convex combination of two positive quantities bounded away from zero, namely 
$\psi^{\prime\prime}_{x_n\xi_n}\left(x^\prime,\theta,\frac{\xi^\prime}{\norm{\xi^\prime}},\tau\right)$, cf. \eqref{eq:6}  and $K$. 

	\item[ii)] In view of  the properties of $\psi$ and $\varphi$,
		\begin{align*}
		\sphi^\prime_t(t,\tau)&=
		\omega^\prime\left(\frac{t}{\norm{\xi^\prime}}\right)\norm{\xi^\prime}^{-1}\left[
					\varphi\left(x^\prime,\frac{t}{\norm{\xi^\prime}},\xi^\prime,\tau\norm{\xi^\prime}\right)-K\cdot t\cdot\tau\right]
		\\
		&+
		\omega\left(\frac{t}{\norm{\xi^\prime}}\right)\cdot
					\psi^\prime_{x_n}\left(x^\prime,\frac{t}{\norm{\xi^\prime}},\xi^\prime,\tau\norm{\xi^\prime}\right)\norm{\xi^\prime}^{-1}
					+\left[1-\omega\left(\frac{t}{\norm{\xi^\prime}}\right)\right]K\cdot \tau
		\\
		&=
		\omega^\prime\left(\frac{t}{\norm{\xi^\prime}}\right)\left[
					\varphi\left(x^\prime,\frac{t}{\norm{\xi^\prime}},\frac{\xi^\prime}{\norm{\xi^\prime}},\tau\right)-
					K\cdot \frac{t}{\norm{\xi^\prime}}\cdot\tau\right]
		\\
		&+
		\omega\left(\frac{t}{\norm{\xi^\prime}}\right)\cdot
					\psi^\prime_{x_n}\left(x^\prime,\frac{t}{\norm{\xi^\prime}},\frac{\xi^\prime}{\norm{\xi^\prime}},\tau\right)
					+\left[1-\omega\left(\frac{t}{\norm{\xi^\prime}}\right)\right]K\cdot \tau.
		\end{align*}
		It is enough to focus on the case $|\tau|\ge1$, since, when $|\tau|\le1$, trivially,
		\[
			\norm{\sphi^\prime_t(t,\tau)}\ge1>\frac{1}{2}\norm{\tau}.
		\]
		Then, writing, for $|\tau|\ge1$,
		\begin{align*}
		\sphi^\prime_t(t,\tau)=
		&\left\{\omega^\prime\left(\frac{t}{\norm{\xi^\prime}}\right)\left[
					\pm\varphi\left(x^\prime,\frac{t}{\norm{\xi^\prime}},\frac{\xi^\prime}{|\tau|\norm{\xi^\prime}},\pm1\right)-
					K\cdot \frac{t}{\norm{\xi^\prime}}\right]\right.
		\\
		&\left.\pm
		\omega\left(\frac{t}{\norm{\xi^\prime}}\right)\cdot
					\psi^\prime_{x_n}\left(x^\prime,\frac{t}{\norm{\xi^\prime}},\frac{\xi^\prime}{|\tau|\norm{\xi^\prime}},\pm1\right)
					+\left[1-\omega\left(\frac{t}{\norm{\xi^\prime}}\right)\right]K\right\} \tau,
		\end{align*}
		we  analyze the coefficient of $\tau$
		\begin{align*}
			B&=\omega^\prime\left(\frac{t}{\norm{\xi^\prime}}\right)\left[
					\pm\varphi\left(x^\prime,\frac{t}{\norm{\xi^\prime}},\frac{\xi^\prime}{|\tau|\norm{\xi^\prime}},\pm1\right)-
					K\cdot \frac{t}{\norm{\xi^\prime}}\right]
			\\
			&\pm
					\omega\left(\frac{t}{\norm{\xi^\prime}}\right)\cdot
					\psi^\prime_{x_n}\left(x^\prime,\frac{t}{\norm{\xi^\prime}},\frac{\xi^\prime}{|\tau|\norm{\xi^\prime}},\pm1\right)
					+\left[1-\omega\left(\frac{t}{\norm{\xi^\prime}}\right)\right]K,
		\end{align*}
		and show $B\ge c_2$ with a constant $c_2>0$, provided that $K>0$ is chosen large enough.

For $t/\norm{\xi'}\ge1$ or $t/\norm{\xi'}\le 1/2$, $B$ is uniformly bounded away from zero in view of the positivity of $K$ and \eqref{eq:6}. This positivity extends, with a uniform lower
bound, to the case of slightly smaller and larger values of $t/\norm{\xi'}$. 
So it remains to consider the case where, for some $\varepsilon>0$,  
$$  \frac12+\varepsilon\le \frac t {\norm{\xi'}} \le 1-\varepsilon.$$ 
On this set,  $\omega(t/\norm{\xi'})\le1-\varepsilon_1$ for some $\varepsilon_1>0$. 
We note that  $\omega'(s)s$ is non-negative and rewrite 
$$B= K\left(\left(1-\omega\left(\frac t{\norm{\xi'}}\right)\right) 
-\omega'\left(\frac t{\norm{\xi'}}\right)\frac t{\norm{\xi'}}\right) +r,
$$ 
where the rest $r$ is bounded and independent of $K$. By making $K$ large, we thus obtain the	positivity of $B$.
\end{enumerate}

\noindent
Ad (P3). We have to estimate from below
\begin{equation}
	\label{eq:3}
	\begin{aligned}
	\sphi^{\prime\prime}_{t\tau}(t,\tau)= \ &
		\omega^\prime\left(\frac{t}{\norm{\xi^\prime}}\right)\norm{\xi^\prime}^{-1}\left[
		\varphi^\prime_{\xi_n}\!\!\left(x^\prime,\frac{t}{\norm{\xi^\prime}},\xi^\prime,\tau\norm{\xi^\prime}\right)\norm{\xi^\prime}-K\cdot t
		\right]
	\\
	&\ 
	+\, \omega\left(\frac{t}{\norm{\xi^\prime}}\right)
	\varphi^{\prime\prime}_{x_n\xi_n}\!\!\left(x^\prime,\frac{t}{\norm{\xi^\prime}},\xi^\prime,\tau\norm{\xi^\prime}\right)
	+K\left[1-\omega\left(\frac{t}{\norm{\xi^\prime}}\right)\right]
	\\
		&=\omega^\prime\left(\frac{t}{\norm{\xi^\prime}}\right)\left[
		\varphi^\prime_{\xi_n}\!\!\left(x^\prime,\frac{t}{\norm{\xi^\prime}},\frac{\xi^\prime}{\norm{\xi^\prime}},\tau\right)-K\cdot \frac{t}{\norm{\xi^\prime}}
		\right]
	\\
	&\  +\omega\left(\frac{t}{\norm{\xi^\prime}}\right)
	\varphi^{\prime\prime}_{x_n\xi_n}\!\!\left(x^\prime,\frac{t}{\norm{\xi^\prime}},\frac{\xi^\prime}{\norm{\xi^\prime}},\tau\right)
	+K\left[1-\omega\left(\frac{t}{\norm{\xi^\prime}}\right)\right],
	\end{aligned}
\end{equation}
where we have used homogeneity in the last equality. 

We now first note that  
$\varphi''_{x_n\xi_n}\!\!\left(x^\prime,\frac{t}{\norm{\xi^\prime}},\xi^\prime,\tau\norm{\xi^\prime}\right)>0$ is bounded away from zero on $U'\times(\R^n\setminus 0)$. 
Indeed, for $|\tau|\le1$ this follows from the positivity of $\varphi''_{x_n\xi_n}$ and the fact that the argument then varies over a bounded set. For $|\tau|\ge 1$ we use the zero-homogeneity of $\varphi''_{x_n\xi_n}$. 
With this in mind, the sum of the last two terms on the right hand side is seen to be bounded away from zero as a convex combination.
In view of the fact that $s\mapsto \omega'(s)s$ is bounded and everywhere non-negative, the first summand will be positive for large $K$.   
This shows the assertion. 	 

\noindent
The proof is complete.
\end{proof}

Before proving Theorem \ref{thm:main2}, we introduce a class of functions which will be useful in the sequel.
\begin{Def}
 \label{def:BS}
 A function $a \in C^{\infty}(\R^{n-1}_{x'} \times \R^{n-1}_{\xi'} \times \R_{x_n} \times \R_{\xi_n})$ belongs to the set
 $BS^{m}(\R^{n-1}, \R^{n-1}; S^{l}(\R))$ 
if, for all $\alpha, \beta \in \N^{n-1}$, and fixed $(x', \xi')$
\[
 \left(\partial_{\xi'}^\alpha \partial_{x'}^\beta a \right)\left(x', \frac{x_n}{\langle \xi' \rangle}, \xi', \xi_n \langle \xi' \rangle\right)  \in S^{l}(\R_{x_n} \times \R_{\xi_n})
\]
and each seminorm can be estimated by $\langle \xi' \rangle^{m-|\alpha|}$. That is, for all $\gamma, \delta$ and compact $K\subseteq\R^n$ there exists
a constant $C_{\gamma, \delta, K}$ such that for all $(x', x_n) \in K$

\[
  \left|\partial_{\xi_n}^{\gamma}\partial_{x_n}^{\delta}\left(\left(\partial_{\xi'}^\alpha \partial_{x'}^\beta a \right)
  \left(x', \frac{x_n}{\langle \xi' \rangle}, \xi', \xi_n \langle \xi' \rangle\right)\right)\right| \leq C_{\gamma, \delta, K} 
 \langle \xi_n \rangle^{l-|\gamma|} \langle \xi' \rangle^{m-|\alpha|} .
\]
\end{Def}

\begin{rem}\label{smbs}
 If $a \in S^m(\R^n, \R^n)$, then
 $a \in BS^m(\R^{n-1},$ $\R^{n-1},$ $S^m(\R))$. This a consequence of the fact that
 $\partial_{\xi'}^{\alpha}\partial_{x'}^{\beta}  a \in S^{m-|\alpha|}(\R^n, \R^n)$
 and of the direct computation
 \begin{equation}
  \label{vero}
  \left|\partial_{\xi_n}^\gamma \partial_{x_n}^\delta a\left(x', \frac{x_n}{\langle \xi' \rangle},
  \xi', \xi_n\langle \xi' \rangle \right)\right|\leq
  C \langle \xi'\rangle^{m} \langle \xi_n \rangle^{m- |\gamma|}, 
 \end{equation}
valid for any $a \in S^m(\R^n, \R^n)$. Moreover, it is clear that $BS$-spaces satisfy a multiplicative property, that is
 \begin{equation} 
 \label{eq:prodbs}
 \begin{aligned}
 BS^{m}(\R^{n-1}, \R^{n-1}; S^l(\R)) \cdot &BS^{m'}(\R^{n-1}, \R^{n-1}; S^{l'}) \subseteq
 \\
\subseteq &BS^{m+m'}(\R^{n-1}, \R^{n-1}, S^{l+l'}(\R)).
 \end{aligned}
 \end{equation}
\end{rem} 
\begin{Lem}
\label{Lem:induc}
Let $a \in S^m(\R^n \times \R^n)$ be a symbol vanishing for $|\xi'|=|\xi_n|=0$ and $\psi$ be a phase function which represents locally at $x_n=0$
an admissible symplectomorphism. Then
\begin{eqnarray}\label{induc}
\lefteqn{\partial_{\xi'}^{\alpha}\partial_{x'}^\beta 
\left(e^{i \psi(x', x_n, \xi', \xi_n)-i \psi_{\partial} (x', \xi')}a(x', x_n,\xi', \xi_n)\right)}
\nonumber\\ 
&=& e^{i \psi(x', x_n, \xi', \xi_n)- i\psi_{\partial}(x', \xi')} \tilde{a}(x', x_n, \xi', \xi_n),
\end{eqnarray}
where $\tilde{a}\left(x', x_n, \xi', \xi_n\right)
 \in BS^{m-|\alpha|}(\R^{n-1}, \R^{n-1}; S^{m+|\beta|}(\R))$.
\end{Lem}

\begin{proof}
The assertion is proven by induction. It is true if $\alpha=\beta=0$ by Remark \ref{smbs}. 
Suppose now that
\eqref{induc} is true for $|\alpha|+|\beta|<t$, $t\in \N$. We show that it holds true for 
 $|\alpha|+|\beta|=t$. If $\alpha\not =0$ we can write, in view if  the inductive hypothesis,
\begin{eqnarray}
 \label{ind}
\lefteqn{
\partial_{\xi'_j} \big( \partial_{\xi'}^{\alpha- 1_j }D_{x'}^\beta e^{i\psi(x',x_n, \xi', \xi_n ) - 
i \psi_{\partial}(x',\xi')} 
a(x', x_n, \xi', \xi_n)\big)}\nonumber\\
&=&\partial_{\xi'_j} \big(e^{i \psi(x', x_n, \xi', \xi_n)- i \psi_{\partial}(x',\xi')} \tilde{a}
 (x',x_n, \xi', \xi_n)\big)\nonumber\\
&=&e^{i\psi(x',x_n, \xi', \xi_n ) - i \psi_{\partial}(x',\xi')} \big( \partial_{\xi'_j}
 (i\psi(x',x_n, \xi', \xi_n )\nonumber\\
 && - i \psi_{\partial}(x',\xi')) \tilde{a} (x',x_n, \xi', \xi_n )+ \partial_{\xi'_j}\tilde{a}(x', x_n, \xi',\xi_n ) \big)\nonumber \\
&=&e^{i\psi(x',x_n, \xi', \xi_n) - i \psi_{\partial}(x',\xi')} \big(b(x', x_n, \xi', \xi_n) \tilde{a}(x', x_n, \xi', \xi_n )\\
&&+ \partial_{\xi'_j}\tilde{a}(x', x_n, \xi',\xi_n  )\big), \nonumber
\end{eqnarray}
where
\begin{equation}
\label{eq:bx}
 b(x', x_n, \xi', \xi_n)= x_n \int_0^1 \partial_{\xi'_j} \partial_{x_n} \psi(x', t x_n, \xi', \xi_n) dt.
\end{equation}
In \eqref{ind}, we have used a Taylor expansion at $x_n=0$, with $b$ in \eqref{eq:bx}  
the corresponding integral remainder, and the fact that $\psi(x',0,\xi', \xi_n)-\psi_\partial(x',\xi')=0$. Now, we have to verify that 
\[
b(x', x_n, \xi', \xi_n )
\tilde{a}\left(x', x_n, \xi', \xi_n \right)
+ (\partial_{\xi'_j}\tilde{a})\left(x', x_n, \xi',\xi_n \right)
\]
belongs to $BS^{m-|\alpha|}(\R^{n-1}, \R^{n-1}, S^{|\beta|}(\R))$. 
By induction, 
$\tilde{a}\left(x', x_n, \xi', \xi_n\right)$ is an element of 
$BS^{m-|\alpha|+1}(\R^{n-1},\R^{n-1}, S^{|\beta|}(\R))$, 
$x_n$ is an element of $ BS^{-1}(\R^{n-1},$ $\R^{n-1},$ $ S^{0}(\R))$,  and 
$\int_0^1 \partial_{\xi'_j} \partial_{x_n} \psi(x', tx_n, \xi', \xi_n) dt$ is a symbol of order zero. So $b$ in \eqref{eq:bx} belongs to $BS^{-1}(\R^{n-1}, \R^{n-1}; S^0(\R))$ by
Remark \ref{smbs}. 
Then, we just apply the multiplicative property \eqref{eq:prodbs}.

If $\alpha=0$, then we have
\begin{eqnarray*}
\lefteqn{\partial_{x'_j}\big(\partial_{x'}^{\beta-1_j} e^{i\psi(x',x_n, \xi', \xi_n ) - i \psi_{\partial}(x',\xi')} 
a(x', x_n, \xi', \xi_n)\big)}\\
&=&\partial_{x'_j}\big( e^{i\psi(x',x_n, \xi', \xi_n ) - i \psi_{\partial}(x',\xi')} 
\tilde{a}(x', x_n, \xi', \xi_n)\big)\\
&=&e^{i\psi(x',x_n, \xi', \xi_n ) - i \psi_{\partial}(x',\xi')}\big( c(x', x_n, \xi', \xi_n)\tilde{a}(x', x_n, \xi', \xi_n)+ 
\partial_{x_j}\tilde{a}(x', x_n, \xi', \xi_n)\big),
\end{eqnarray*} 
where 
\[
 c(x', x_n, \xi', \xi_n)= x_n \int_0^1 \partial_{x'_j} \partial_{x_n} \psi(x', t x_n, \xi', \xi_n) dt
\]
is the remainder in the Taylor expansion of $\partial_{x'_j}((\psi(x', x_n, \xi', \xi_n)- 
\psi_{\partial}(x', \xi'))$ at $x_n=0$.
Again, by the inductive hypothesis,  $\tilde{a}(x', x_n, \xi', \xi_n )$ 
$ \in BS^{m}$ $(\R^{n-1}, $ $\R^{n-1}, $ $S^{m+|\beta|-1}(\R))$, while
$c \in BS^{0}(\R^{n-1}, \R^{n-1};S^1(\R))$. Thus, applying the multiplicative property
 \eqref{eq:prodbs}, the assertion is proven.
\end{proof}

\begin{proof}[Proof of Theorem \ref{thm:main2}.]
We have to show that, for all $\alpha, \beta \in \N^{n-1}$ and $l, s \in \N$, there exist $\gamma, \delta \in \N$ such that
\begin{equation}
\label{eq:intgen}
 \sup\left|x_n^{l} \partial_{x_n}^{s} \left(\kappa_{\langle \xi' \rangle^{-1} }\partial_{x'}^{\alpha}\partial_{\xi'}^{\beta} A_n^\chi 
 \kappa_{\langle \xi' \rangle}u\right)\right|\leq 
 C_{\beta, \gamma} p_{\gamma, \delta}(u)\langle\xi'\rangle^{m-|\alpha|},
\end{equation}
$p_{\gamma, \delta}$ being seminorms of $\Si(\R)$.
As we have already pointed out, the hypotheses imply that we can apply Theorem \ref{thm:main}, 
that is the Fourier integral operator operator in \eqref{eq:intgen} has a regular SG phase function. 

First, let us suppose $s=0$.
Lemma \ref{Lem:induc} implies that, for all $\alpha$, $\beta$, the operator in \eqref{eq:intgen}
is a Fourier integral operator with SG phase function and a symbol $a \in BS^{m-|\alpha|}(\R^{n-1}, \R^{n-1}; S^{m+\beta}(\R))$. Actually, since $a$ has compact support in the $x_n$ variable, we could write $a \in BS^{m+|\alpha|}(\R^{n-1}, \R^{n-1}, S^{m+|\beta|,0}(\R))$,
where $S^{m+|\beta|,0}(\R, \R)$ is the class of SG symbols of order $(m+|\beta|,0)$, see \cite{CO99}. 
Therefore, the expression in \eqref{eq:intgen} is equivalent to the evaluation of the $\Si(\R)$ continuity of an SG Fourier integral operator with an SG symbol of order $(m+|\beta|,0)$, such that all its seminorms are bounded by a multiple of $\langle \xi'\rangle^{m-|\alpha|}$. 
Hence,  the inequality \eqref{eq:intgen} is a consequence of the theory developed in \cite{CO99}.

If $s>0$, then it is enough to notice that the derivative of the phase function \eqref{eq:1}
is a symbol in $BS^0(\R^{n-1},\R^{n-1}, S^{1,0}(\R) )$, while the derivatives of the
symbol are again of the same type as above. Therefore, also in this case we can use the theory of SG Fourier integral operator developed in \cite{CO99}, recalling the multiplicative property \eqref{eq:prodbs},
extended to the case of SG symbols.

Note that, by a completely similar argument, the same result holds true for the transpose operator
$(A_n^\chi)^t$. Hence we have, by duality, that
$A_n^\chi$ can be extended as an operator-valued symbol in $S^m(\R^{n-1}, \R^{n-1}; \Si'(\R), \Si'(\R))$,
showing that also the second part of \eqref{eq:opsymb} holds true, and completing the proof.
\end{proof}

\begin{rem}
Observe that  $i:L^2(\R) \to \Si'(\R)$ and  $e^+:L^2(\R_+)\to L^2(\R)$ can both be interpreted 
as  operator-valued symbols of order $0$. This implies that
\[
 	 A_n^\chi e^+ \in S^m (\R^{n-1}, \R^{n-1}; L^2(\R_+), \Si'(\R)).
\]
Also $r^+:L^2(\R)\to L^2(\R_+)$ can be considered as an operator-valued symbol of order $0$. Suppose the local symbol $a\in S^0$ of $A^\chi$ is compactly supported with respect to the $x$-variable. Then $a$ can be interpreted as an $\SG$-symbol of order $0,0$. Recalling that $\SG$ Fourier integral operators with
regular phase function and symbol of order $0,0$ are $L^2(\R)$-continuous, see \cite{CO99}, we conclude that in this case, for each $(x^\prime,\xi^\prime)$,
\[
 	A_n^+= r^+ A_n^\chi e^+ \in \mathcal{L} (L^2(\R_+), L^2(\R_+)).
\]
\end{rem}


\begin{thebibliography}{10}

\bibitem{BCS14}
U.~Battisti, S~Coriasco, E.~Schrohe.
\newblock Fourier integral operators of Boutet de Monvel type.
\newblock In preparation.

\bibitem{BU71}
L.~Boutet~de Monvel.
\newblock Boundary problems for pseudo-differential operators.
\newblock {\em Acta Math.}, 126(1-2):11--51, 1971.

%
\bibitem{DAS01}
A.~Cannas~da Silva.
\newblock {\em Lectures on Symplectic Geometry}, volume 1764 of {\em Lecture
  Notes in Mathematics}.
\newblock Springer-Verlag, Berlin, 2001.

%
\bibitem{CO99}
S.~Coriasco.
\newblock Fourier integral operators in {SG} classes. {I}. {C}omposition
  theorems and action on {SG} {S}obolev spaces.
\newblock {\em Rend. Sem. Mat. Univ. Politec. Torino},
  1999, 57(4):249--302 (2002).

\bibitem{CO99b}
S.~Coriasco.
\newblock {Fourier Integral Operators in $SG$ classes II. Application to $SG$
  Hyperbolic Cauchy Problems}.
\newblock {\em Ann. Univ. Ferrara}, 1998, 47:81--122 (1999).

\bibitem{DS76}
J.~J. Duistermaat and I.~M. Singer.
\newblock Order-preserving isomorphisms between algebras of pseudo-differential
  operators.
\newblock {\em Comm. Pure Appl. Math.}, 29(1):39--47, 1976.

%
%
%

\bibitem{GR87}
G.~Grubb.
\newblock Complex powers of pseudodifferential boundary value problems with the
  transmission property.
\newblock In {\em Pseudodifferential operators ({O}berwolfach, 1986)}, volume
  1256 of {\em Lecture Notes in Math.}, pages 169--191. Springer, Berlin, 1987.

%
\bibitem{GH90}
G.~Grubb and L.~H{\"o}rmander.
\newblock The transmission property.
\newblock {\em Math. Scand.}, 67(2):273--289, 1990.

%
\bibitem{HP77}
A.~Hirschowitz and A.~Piriou.
\newblock La propri\'et\'e de transmission pour les distributions de {F}ourier;
  application aux lacunes.
\newblock In {\em S\'eminaire {G}oulaouic-{S}chwartz (1976/1977), \'{E}quations
  aux d\'eriv\'ees partielles et analyse fonctionnelle, {E}xp. {N}o. 14},
  page~19. Centre Math., \'Ecole Polytech., Palaiseau, 1977.

\bibitem{HO03}
L.~H{\"o}rmander.
\newblock {\em The analysis of linear partial differential operators. {III}}.
\newblock Classics in Mathematics. Springer, Berlin, 2007.
\newblock Pseudo-differential operators, Reprint of the 1994 edition.

\bibitem{HO04}
L.~H{\"o}rmander.
\newblock {\em The analysis of linear partial differential operators. {IV}}.
\newblock Classics in Mathematics. Springer, Berlin, 2007.
\newblock Pseudo-differential operators, Reprint of the 1994 edition.

\bibitem{LSV94}
A.~Laptev, Y.~Safarov, and D.~Vassiliev.
\newblock On global representation of {L}agrangian distributions and solutions
  of hyperbolic equations.
\newblock {\em Comm. Pure Appl. Math.}, 47(11):1411--1456, 1994.

%

\bibitem{ME81}
R.~B. Melrose.
\newblock Transformation of boundary problems.
\newblock {\em Acta Math.}, 147(3-4):149--236, 1981.

%
%
\bibitem{RS85}
S.~Rempel and B.-W. Schulze.
\newblock {\em Index theory of elliptic boundary problems}.
\newblock North Oxford Academic Publishing Co. Ltd., London, 1985.
\newblock Reprint of the 1982 edition.

\bibitem{SC01}
E.~Schrohe.
\newblock A short introduction to {B}outet de {M}onvel's calculus.
\newblock In {\em Approaches to singular analysis ({B}erlin, 1999)}, volume 125
  of {\em Oper. Theory Adv. Appl.}, pages 85--116. Birkh\"auser, Basel, 2001.

\bibitem{SC98}
B.-W. Schulze.
\newblock {\em Boundary Value Problems and Singular Pseudo-differential
  Operators}.
\newblock Pure and Applied Mathematics (New York). John Wiley \& Sons Ltd.,
  Chichester, 1998.

%
\bibitem{WE75}
A.~Weinstein.
\newblock Fourier integral operators, quantization, and the spectra of
  {R}iemannian manifolds.
\newblock In {\em G\'eom\'etrie symplectique et physique math\'ematique
  ({C}olloq. {I}nternat. {CNRS}, {N}o. 237, {A}ix-en-{P}rovence, 1974)}, pages
  289--298. \'Editions Centre Nat. Recherche Sci., Paris, 1975.
\newblock With questions by W. Klingenberg and K. Bleuler and replies by the
  author.

%
\bibitem{YA97}
K.~Yagdjian.
\newblock {\em The {C}auchy Problem for Hyperbolic Operators. Multiple characteristics, Micro-local approach}, volume~12 of
  {\em Mathematical Topics}.
\newblock Akademie Verlag, Berlin, 1997.

%
\end{thebibliography}
\end{document}